\documentclass{amsart}
\usepackage{amssymb}
\usepackage{amsfonts}

\newtheorem{theorem}{Theorem}
\theoremstyle{plain}

\numberwithin{equation}{section}
\input{tcilatex}

\begin{document}
\title[Inequality for Twice Differentiable Functions]{On the Ostrowski-Gr%
\"{u}ss Type Inequality for Twice Differentiable Functions}
\author{$^{\blacktriangledown }$M. Emin \"{O}zdemir}
\address{$^{\blacktriangledown }$Ataturk University, K. K. Education
Faculty, Department of Mathematics, 25640, Kampus, Erzurum, Turkey}
\email{emos@atauni.edu.tr}
\author{$^{\spadesuit ,\bigstar }$Ahmet Ocak Akdemir}
\address{$^{\bigstar }$A\u{g}r\i\ \.{I}brahim \c{C}e\c{c}en University,
Faculty of Science and Arts, Department of Mathematics, 04100, A\u{g}r\i ,
Turkey}
\email{ahmetakdemir@agri.edu.tr}
\author{$^{\blacktriangledown }$Erhan Set}
\address{$^{\blacktriangledown }$Ataturk University, K. K. Education
Faculty, Department of Mathematics, 25640, Kampus, Erzurum, Turkey}
\email{erhanset@yahoo.com}
\date{October 20, 2010}
\keywords{Ostrowski-Gr\"{u}ss Inequality\\
$^{\spadesuit }Corresponding$ $Author$}

\begin{abstract}
In this paper, we obtained some new Ostrowski-Gr\"{u}ss type inequalities
contains twice differentiable functions.
\end{abstract}

\maketitle

\section{Introduction}

In \cite{OST}, Ostrowski proved the following inequality.

\begin{theorem}
Let $f:I\rightarrow 
\mathbb{R}
,$ where $I\subset 
\mathbb{R}
$ is an interval, be a mapping differentiable in the interior of $I$ and $%
a,b\in I^{o},a<b.$ If $\left\vert f^{\prime }\right\vert \leq M,\forall t\in
\lbrack a,b],$ then we have%
\begin{equation}
\left\vert f(x)-\frac{1}{b-a}\dint\limits_{a}^{b}f(t)dt\right\vert \leq %
\left[ \frac{1}{4}+\frac{\left( x-\frac{a+b}{2}\right) ^{2}}{\left(
b-a\right) ^{2}}\right] \left( b-a\right) M,  \label{1.1}
\end{equation}%
for $x\in \lbrack a,b].$
\end{theorem}

In the past several years there has been considerable interest in the study
of Ostrowski type inequalities. In \cite{ME}, \"{O}zdemir et al. proved
Ostrowski's type inequalities for $\left( \alpha ,m\right) -$convex
functions and in \cite{MZ}, an Ostrowski type inequality was given by Sar\i
kaya. However, some new types inequalities are established, for example
inequalities of Ostrowski-Gr\"{u}ss type and inequalities of
Ostrowski-Chebyshev type. In \cite{MP}, Milovanovic and Pecaric gave
generalization of Ostrowski's inequality and some related applications. It
was for the first time that Ostrowski-Gr\"{u}ss type inequality was given by
Dragomir and Wang in \cite{SW}. In \cite{MPU}, Matic et al., generalized and
improved this inequality. For generalizations, improvements and recent
results see the papers \cite{MP}, \cite{CHENG}, \cite{MPU}, \cite{SW} and 
\cite{ANA}. Recently, in \cite{UJE}, Ujevic proved following theorems;

\begin{theorem}
Let $f:I\rightarrow 
\mathbb{R}
,$ where $I\subset 
\mathbb{R}
$ is an interval, be a mapping differentiable in the interior of $I$ and $%
a,b\in I^{o},a<b.$ If there exist constants $\gamma ,\Gamma \in $ $%
\mathbb{R}
$ such that $\gamma \leq f^{\prime }(t)\leq \Gamma ,\forall t\in \lbrack
a,b] $ and $f^{\prime }\in L_{1}[a,b]$, then we have%
\begin{equation}
\left\vert f(x)-\left( x-\frac{a+b}{2}\right) \frac{f(b)-f(a)}{b-a}-\frac{1}{%
b-a}\dint\limits_{a}^{b}f(t)dt\right\vert \leq \frac{\left( b-a\right) }{2}%
\left( S-\gamma \right)  \label{1.2}
\end{equation}%
and%
\begin{equation}
\left\vert f(x)-\left( x-\frac{a+b}{2}\right) \frac{f(b)-f(a)}{b-a}-\frac{1}{%
b-a}\dint\limits_{a}^{b}f(t)dt\right\vert \leq \frac{\left( b-a\right) }{2}%
\left( \Gamma -S\right) ,  \label{1.3}
\end{equation}%
where $S=\frac{f(b)-f(a)}{b-a}.$
\end{theorem}

\begin{theorem}
Let $f:I\rightarrow 
\mathbb{R}
,$ where $I\subset 
\mathbb{R}
$ is an interval, be a twice continuously differentiable mapping in the
interior of $I$ with $f^{\prime \prime }\in L_{2}[a,b]$ and $a,b\in
I^{o},a<b.$ Then we have%
\begin{equation}
\left\vert f(x)-\left( x-\frac{a+b}{2}\right) \frac{f(b)-f(a)}{b-a}-\frac{1}{%
b-a}\dint\limits_{a}^{b}f(t)dt\right\vert \leq \frac{\left( b-a\right) ^{%
\frac{3}{2}}}{2\pi \sqrt{3}}\left\Vert f^{\prime \prime }\right\Vert _{2},
\label{1.4}
\end{equation}%
for $x\in \lbrack a,b].$
\end{theorem}

The main purpose of this paper is to prove Ostrowski-Gr\"{u}ss type
inequalities similar to above but now for involving twice differentiable
mappings.

\section{Main Results}

\begin{theorem}
Let $f:I\rightarrow 
\mathbb{R}
,$ where $I\subset 
\mathbb{R}
$ is an interval, be a twice differentiable mapping in the interior of $I$
and $a,b\in I^{o},a<b.$ If there exist constants $\gamma ,\Gamma \in $ $%
\mathbb{R}
$ such that $\gamma \leq f^{\prime \prime }(t)\leq \Gamma ,\forall t\in
\lbrack a,b]$ and $f^{\prime \prime }\in L_{2}[a,b]$, then we have%
\begin{eqnarray}
&&\left\vert f(x)-xf^{\prime }(x)-\frac{a^{2}f^{\prime }(a)-b^{2}f^{\prime
}(b)}{2\left( b-a\right) }\right.  \label{2.1} \\
&&\left. -\left( \frac{x^{2}}{2}-\frac{a^{2}+ab+b^{2}}{3}\right) \frac{%
f^{\prime }(b)-f^{\prime }(a)}{b-a}-\frac{1}{b-a}\dint\limits_{a}^{b}f(t)dt%
\right\vert  \notag \\
&\leq &\frac{\left( b-a\right) ^{2}}{3}\left( S-\gamma \right)  \notag
\end{eqnarray}%
and%
\begin{eqnarray}
&&\left\vert f(x)-xf^{\prime }(x)-\frac{a^{2}f^{\prime }(a)-b^{2}f^{\prime
}(b)}{2\left( b-a\right) }\right.  \label{2.2} \\
&&\left. -\left( \frac{x^{2}}{2}-\frac{a^{2}+ab+b^{2}}{3}\right) \frac{%
f^{\prime }(b)-f^{\prime }(a)}{b-a}-\frac{1}{b-a}\dint\limits_{a}^{b}f(t)dt%
\right\vert  \notag \\
&\leq &\frac{\left( b-a\right) ^{2}}{3}\left( \Gamma -S\right)  \notag
\end{eqnarray}%
where $S=\frac{f^{\prime }(b)-f^{\prime }(a)}{b-a}.$
\end{theorem}

\begin{proof}
We can define $K(x,t)$ as following 
\begin{equation*}
K(x,t)=\left\{ 
\begin{array}{c}
\frac{t}{2}\left( t-2a\right) ,\text{ \ \ \ \ }t\in \lbrack a,x] \\ 
\\ 
\frac{t}{2}\left( t-2b\right) ,\text{ \ \ \ \ }t\in (x,b]%
\end{array}%
\right.
\end{equation*}%
Integrating by parts, we have%
\begin{eqnarray}
&&\frac{1}{b-a}\dint\limits_{a}^{b}K(x,t)f^{\prime \prime }(t)dt  \notag \\
&=&\frac{1}{b-a}\left[ \dint\limits_{a}^{x}\frac{t}{2}\left( t-2a\right)
f^{\prime \prime }(t)dt+\dint\limits_{x}^{b}\frac{t}{2}\left( t-2b\right)
f^{\prime \prime }(t)dt\right]  \label{2.3} \\
&=&xf^{\prime }(x)-f(x)+\frac{a^{2}f^{\prime }(a)-b^{2}f^{\prime }(b)}{%
2\left( b-a\right) }+\frac{1}{b-a}\dint\limits_{a}^{b}f(t)dt  \notag
\end{eqnarray}%
It is easy to see that%
\begin{equation}
\frac{1}{b-a}\dint\limits_{a}^{b}K(x,t)dt=\frac{x^{2}}{2}-\frac{%
a^{2}+ab+b^{2}}{3}  \label{2.4}
\end{equation}%
and%
\begin{equation}
\dint\limits_{a}^{b}f^{\prime \prime }(t)dt=f^{\prime }(b)-f^{\prime }(a)
\label{2.5}
\end{equation}%
Using (\ref{2.3}), (\ref{2.4}) and (\ref{2.5}), we get%
\begin{eqnarray*}
&&xf^{\prime }(x)-f(x)+\frac{a^{2}f^{\prime }(a)-b^{2}f^{\prime }(b)}{%
2\left( b-a\right) }-\left( \frac{x^{2}}{2}-\frac{a^{2}+ab+b^{2}}{3}\right) 
\frac{f^{\prime }(b)-f^{\prime }(a)}{b-a}+\frac{1}{b-a}\dint%
\limits_{a}^{b}f(t)dt \\
&=&\frac{1}{b-a}\dint\limits_{a}^{b}K(x,t)f^{\prime \prime }(t)dt-\frac{1}{%
\left( b-a\right) ^{2}}\dint\limits_{a}^{b}f^{\prime \prime
}(t)dt\dint\limits_{a}^{b}K(x,t)dt
\end{eqnarray*}%
We denote%
\begin{equation*}
R_{n}(x)=\frac{1}{b-a}\dint\limits_{a}^{b}K(x,t)f^{\prime \prime }(t)dt-%
\frac{1}{\left( b-a\right) ^{2}}\dint\limits_{a}^{b}f^{\prime \prime
}(t)dt\dint\limits_{a}^{b}K(x,t)dt
\end{equation*}%
If we write $R_{n}(x)$ as following with $C\in 
\mathbb{R}
$ which is an arbitrary constant, then we have%
\begin{equation}
R_{n}(x)=\frac{1}{b-a}\dint\limits_{a}^{b}\left( f^{\prime \prime
}(t)-C\right) \left[ K(x,t)-\frac{1}{b-a}\dint\limits_{a}^{b}K(x,s)ds\right]
dt  \label{2.6}
\end{equation}%
We know that%
\begin{equation}
\dint\limits_{a}^{b}\left[ K(x,t)-\frac{1}{b-a}\dint\limits_{a}^{b}K(x,s)ds%
\right] dt=0  \label{2.7}
\end{equation}%
So, if we choose $C=\gamma $ in (\ref{2.6}). Then we get%
\begin{equation*}
R_{n}(x)=\frac{1}{b-a}\dint\limits_{a}^{b}\left( f^{\prime \prime
}(t)-\gamma \right) \left[ K(x,t)-\frac{1}{b-a}\dint\limits_{a}^{b}K(x,s)ds%
\right] dt
\end{equation*}%
and%
\begin{equation}
\left\vert R_{n}(x)\right\vert \leq \frac{1}{b-a}\max_{t\in \lbrack
a,b]}\left\vert K(x,t)-\left( \frac{x^{2}}{2}-\frac{a^{2}+ab+b^{2}}{3}%
\right) \right\vert \dint\limits_{a}^{b}\left\vert f^{\prime \prime
}(t)-\gamma \right\vert dt  \label{2.8}
\end{equation}%
Since%
\begin{equation*}
\max_{t\in \lbrack a,b]}\left\vert K(x,t)-\left( \frac{x^{2}}{2}-\frac{%
a^{2}+ab+b^{2}}{3}\right) \right\vert =\frac{\left( b-a\right) ^{2}}{3}
\end{equation*}%
and%
\begin{eqnarray*}
\dint\limits_{a}^{b}\left\vert f^{\prime \prime }(t)-\gamma \right\vert dt
&=&f^{\prime }(b)-f^{\prime }(a)-\gamma \left( b-a\right) \\
&=&\left( S-\gamma \right) \left( b-a\right)
\end{eqnarray*}%
from (\ref{2.8}), we have%
\begin{equation}
\left\vert R_{n}(x)\right\vert \leq \frac{\left( b-a\right) ^{2}}{3}\left(
S-\gamma \right)  \label{2.9}
\end{equation}%
which gives (\ref{2.1}).

Second, if we choose $C=\Gamma $ in (\ref{2.6}) and by a similar argument we
get%
\begin{equation}
\left\vert R_{n}(x)\right\vert \leq \frac{1}{b-a}\max_{t\in \lbrack
a,b]}\left\vert K(x,t)-\left( \frac{x^{2}}{2}-\frac{a^{2}+ab+b^{2}}{3}%
\right) \right\vert \dint\limits_{a}^{b}\left\vert f^{\prime \prime
}(t)-\Gamma \right\vert dt  \label{2.10}
\end{equation}%
and%
\begin{eqnarray}
\dint\limits_{a}^{b}\left\vert f^{\prime \prime }(t)-\Gamma \right\vert dt
&=&\Gamma \left( b-a\right) -f^{\prime }(b)+f^{\prime }(a)  \label{2.11} \\
&=&\left( \Gamma -S\right) \left( b-a\right)  \notag
\end{eqnarray}%
From (\ref{2.10}) and (\ref{2.11}), we get (\ref{2.2}).
\end{proof}

\begin{theorem}
Let $f:I\rightarrow 
\mathbb{R}
,$ where $I\subset 
\mathbb{R}
$ is an interval, be a twice continuously differentiable mapping in the
interior of $I$ with $f^{\prime \prime }\in L_{2}[a,b]$ and $a,b\in
I^{o},a<b.$ Then we have%
\begin{eqnarray}
&&\left\vert f(x)-xf^{\prime }(x)-\frac{a^{2}f^{\prime }(a)-b^{2}f^{\prime
}(b)}{2\left( b-a\right) }\right.  \label{2.12} \\
&&\left. -\left( \frac{x^{2}}{2}-\frac{a^{2}+ab+b^{2}}{3}\right) \frac{%
f^{\prime }(b)-f^{\prime }(a)}{b-a}-\frac{1}{b-a}\dint\limits_{a}^{b}f(t)dt%
\right\vert  \notag \\
&\leq &\frac{\left( b-a\right) ^{2}}{3}\left( S-f^{\prime \prime }\left( 
\frac{a+b}{2}\right) \right)  \notag
\end{eqnarray}%
where $S=\frac{f^{\prime }(b)-f^{\prime }(a)}{b-a}.$
\end{theorem}

\begin{proof}
$R_{n}(x)$ be defined as above, we can write%
\begin{equation}
R_{n}(x)=\frac{1}{b-a}\dint\limits_{a}^{b}\left( f^{\prime \prime
}(t)-C\right) \left[ K(x,t)-\frac{1}{b-a}\dint\limits_{a}^{b}K(x,s)ds\right]
dt  \label{2.13}
\end{equation}%
If we choose $C=f^{\prime \prime }\left( \frac{a+b}{2}\right) $, we get%
\begin{equation*}
\left\vert R_{n}(x)\right\vert \leq \frac{1}{b-a}\max_{t\in \lbrack
a,b]}\left\vert K(x,t)-\left( \frac{x^{2}}{2}-\frac{a^{2}+ab+b^{2}}{3}%
\right) \right\vert \dint\limits_{a}^{b}\left\vert f^{\prime \prime
}(t)-f^{\prime \prime }\left( \frac{a+b}{2}\right) \right\vert dt
\end{equation*}%
By a simple computation, we get the required result.
\end{proof}


\begin{thebibliography}{99}
\bibitem{OST} A. Ostrowski, \"{U}ber die Absolutabweichung einer
differentierbaren Funktion von ihren Integralmittelwert, Comment. Math.
Helv., 10, 226-227, (1938).

\bibitem{MP} G.V. Milovanovic and J. E. Pecaric, On generalization of the
inequality of A. Ostrowski and some related applications, Univ. Beograd
Publ. Elektrotehn. Fak. Ser. Mat.

Fiz. (544-576), 155-158, (1976).

\bibitem{CHENG} X.L. Cheng, Improvement of some Ostrowski-Gr\"{u}ss type
inequalities, Computers Math. Applic., 42 (1/2), 109-114, (2001).

\bibitem{MPU} M. Matic, J. Pecaric and N. Ujevic, Improvement and further
generalizationof some inequalities of Ostrowski-Gr\"{u}ss type, Comput.
Math. Appl., 39 (3/4), 161-175, (2000).

\bibitem{SW} S.S. Dragomir and S. Wang, An inequality of Ostrowski-Gr\"{u}s
type and its applications to the estimation of error bounds for some special
means and for some numeraical quadrature rules, Comput. Math. Appl., 33
(11), 15-20, (1997).

\bibitem{ANA} G.A. Anastassiou, Ostrowski type inequalities, Proc. Amer.
Math. Soc., 123, 3775-3781, (1995).

\bibitem{UJE} N. Ujevic, New bounds for the first inequality of Ostrowski-Gr%
\"{u}ss type and applications, Comput. Math. Appl., 46, 421-427, (2003).

\bibitem{ME} M.E. \"{O}zdemir, H. Kavurmac\i\ and E. Set, Ostrowski's type
inequalities for $\left( \alpha ,m\right) -$convex functions, KYUNGPOOK
Math. J., 50 (2010), 371-378.

\bibitem{MZ} M. Z. Sar\i kaya, On the Ostrowski type integral inequality,
Acta Math. Univ. Comenianee, Vol. LXXIX, 1 (2010), 129-134.
\end{thebibliography}
\end{document}